\documentclass [leqno] {amsart}
  
\usepackage {amssymb}
\usepackage {amsthm}
\usepackage {amscd}

\usepackage[all]{xy}
\usepackage {hyperref} 

\begin{document}

\theoremstyle{plain}
\newtheorem{proposition}[subsubsection]{Proposition}
\newtheorem{lemma}[subsubsection]{Lemma}
\newtheorem{corollary}[subsubsection]{Corollary}
\newtheorem{thm}[subsubsection]{Theorem}
\newtheorem{introthm}{Theorem}
\newtheorem*{thm*}{Theorem}
\newtheorem{conjecture}[subsubsection]{Conjecture}
\newtheorem{question}[subsubsection]{Question}

\theoremstyle{definition}
\newtheorem{definition}[subsubsection]{Definition}
\newtheorem{notation}[subsubsection]{Notation}
\newtheorem{condition}[subsubsection]{Condition}
\newtheorem{example}[subsubsection]{Example}

\theoremstyle{remark}
\newtheorem{remark}[subsubsection]{Remark}

\numberwithin{equation}{subsection}

\newcommand{\eq}[2]{\begin{equation}\label{#1}#2 \end{equation}}
\newcommand{\ml}[2]{\begin{multline}\label{#1}#2 \end{multline}}
\newcommand{\mlnl}[1]{\begin{multline*}#1 \end{multline*}}
\newcommand{\ga}[2]{\begin{gather}\label{#1}#2 \end{gather}}

\newcommand{\arir}{\ar@{^{(}->}}
\newcommand{\aril}{\ar@{_{(}->}}
\newcommand{\are}{\ar@{>>}}

\newcommand{\xr}[1] {\xrightarrow{#1}}
\newcommand{\xl}[1] {\xleftarrow{#1}}
\newcommand{\lra}{\longrightarrow}
\newcommand{\inj}{\hookrightarrow}

\newcommand{\mf}[1]{\mathfrak{#1}}
\newcommand{\mc}[1]{\mathcal{#1}}

\newcommand{\CH}{{\rm CH}}
\newcommand{\Gr}{{\rm Gr}}
\newcommand{\codim}{{\rm codim}}
\newcommand{\cd}{{\rm cd}}
\newcommand{\Spec} {{\rm Spec}}
\newcommand{\supp} {{\rm supp}}
\newcommand{\Hom} {{\rm Hom}}
\newcommand{\End} {{\rm End}}
\newcommand{\id}{{\rm id}}
\newcommand{\Aut}{{\rm Aut}}
\newcommand{\sHom}{{\rm \mathcal{H}om}}
\newcommand{\Tr}{{\rm Tr}}


\renewcommand{\P} {\mathbb{P}}
\newcommand{\Z} {\mathbb{Z}}
\newcommand{\Q} {\mathbb{Q}}
\newcommand{\C} {\mathbb{C}}
\newcommand{\F} {\mathbb{F}}

\newcommand{\OO}{\mathcal{O}}
\newcommand{\Fr}{{\rm Frob}}
\newcommand{\coker}{{\rm coker}}
\newcommand{\perf}{{\rm perf}}
\newcommand{\Ext}{{\rm Ext}}
\newcommand{\BH}{{\rm BH}}

\title{On the Beilinson-Hodge conjecture for $H^2$ and rational varieties}

\author{Andre Chatzistamatiou}
\address{Fachbereich Mathematik \\ Universit\"at Duisburg-Essen \\ 45117 Essen, Germany}
\email{a.chatzistamatiou@uni-due.de}

\thanks{This work has been supported by the SFB/TR 45 ``Periods, moduli spaces and arithmetic of algebraic varieties''}

\begin{abstract}
The Beilinson-Hodge conjecture asserts the surjectivity of the cycle map 
$$
H^n_M(X,\Q(n)) \to \Hom_{{\rm MHS}}(\Q(-n),H^{n}(X,\Q))
$$
for all integers $n\geq 1$ and every smooth complex algebraic variety $X$.
For $n=2$, we prove the conjecture if $X$ is rational. 
\end{abstract}

\maketitle

\tableofcontents

\section*{Introduction}
The Beilinson-Hodge conjecture ($\BH(X,n)$) asserts the surjectivity of the 
cycle map 
$$
\bar{c}_{n,n}:  H^n_M(X,\Q(n)) \xr{} \Hom_{{\rm MHS}}(\Q(-n),H^{n}(X,\Q))
$$
for all integers $n\geq 1$ and every smooth complex algebraic variety $X$.
Informally speaking, the conjecture means that every complex holomorphic $n$-form 
with logarithmic poles along the boundary divisor of every compactification
of $X$ and rational cohomology class comes from a meromorphic form of the
shape 
$$
\frac{1}{(2\pi i)^n} \sum_{j} m_j \cdot \frac{df_{j1}}{f_{j1}}\wedge \dots \wedge 
\frac{df_{jn}}{f_{jn}}
$$  
with $f_{jk}\in \C(X)^*$ and $m_j\in \Q$.

It is well-known that the conjecture holds for $n=1$ (see, for example, \cite[Proposition~2.12]{EV}). For $n\geq 2$, Asakura and Saito provided evidence 
for the conjecture by studying the Noether-Lefschetz locus of Beilinson-Hodge
cycles (see \cite{AS1},\cite{AS2},\cite{AS3}). By work of Arapura and Kumar, $\BH(X,n)$ is known to hold for every $n$ provided that $X$ 
is a semi-abelian variety or a product of curves \cite{AK}.

In our paper we consider only the case $n=2$, and make two observations.
First, for a smooth and projective variety $X$ the Beilinson-Hodge 
conjecture $\BH(\eta,2)$ for the generic point $\eta$ of $X$ is
equivalent to the injectivity of the cycle map 
$$
\frac{H^3_M(X,\Q(2))}{H^1_M(X,\Q(1))\cdot H^2_M(X,\Q(1))}\xr{} 
\frac{H^3_{\mc{H}}(X,\Q(2))}{H^1_{\mc{H}}(X,\Q(1))\cdot H^2_{\mc{H}}(X,\Q(1))}
$$   
to absolute Hodge cohomology (Proposition \ref{proposition-ex-seq-general}). The left hand side is called the group of 
\emph{indecomposable} cycles and has been studied
by M\"uller-Stach \cite{MS}. In general, indecomposable cycles exist, but
the image via the cycle class map is a countable group.
By $\BH(\eta,2)$ we mean the surjectivity of the cycle class map
$$
H^2_{M}(\C(X),\Q(2))\xr{} \varinjlim_{U\subset X}\Hom_{{\rm MHS}}(\Q(-2),H^{2}(U,\Q)),
$$ 
where $U$ runs over all open subsets of $X$, and $\C(X)$ denotes the function field.

The second observation is that if $X$ satisfies $H^1(X,\C)=0$ then 
$\BH(U,2)$ for all open sets $U\subset X$ is equivalent to $\BH(\eta,2)$
(Proposition \ref{proposition-reduction-to-generic-point}).
The statement makes perfect sense when $H^1(X,\C)\neq 0$, but we can prove
it only in the case $H^1(X,\C)=0$.
   
Combining the two observations we obtain the main theorem of the paper.
\begin{thm*}[cf.~Theorem \ref{main-thm}]\label{main-thm-intro} 
Let $X$ be smooth and connected. Let $\bar{X}$ be a smooth compactification of $X$.
We denote by $\CH_0(\bar{X})\otimes_{\Z}\Q$ the Chow group of zero cycles on $\bar{X}$. 
If $\deg:\CH_0(\bar{X})\otimes_{\Z}\Q \xr{} \Q$ is an isomorphism then $BH(X,2)$ holds.
\end{thm*}

For the proof we use a theorem of Bloch and Srinivas \cite{BS} which states
that the indecomposable cycles vanish whenever the assumptions of our 
theorem are satisfied.


\subsection*{Acknowledgements} 
It is a pleasure to thank H\'el\`ene Esnault for her strong encouragement. 
I thank Donu Arapura, Florian Ivorra and  Manish Kumar for  useful discussions.

\section{Cycle class to absolute Hodge cohomology}

\subsection{Higher Chow groups and absolute Hodge cohomology} 
Let $X$ be a smooth algebraic variety over the complex numbers. 

Absolute Hodge cohomology was introduced by Beilinson 
(\cite{Bei}, cf.~\cite[\textsection2]{JA}). Beilinson constructs for 
every complex algebraic variety $X$ an object $\underline{R\Gamma}(X,\Q)$
in the derived category of mixed Hodge structures $D^b(MHS)$ such that 
$$
H^i(\underline{R\Gamma}(X,\Q))=H^i(X,\Q) \quad \text{for all $i$.}
$$   
Absolute Hodge cohomology $H_{\mc{H}}^{\bullet}(X,\Q(\bullet))$ is defined as follows: 
$$
H_{\mc{H}}^{q}(X,\Q(p))=\Hom_{D^b(MHS)}(\Q,\underline{R\Gamma}(X,\Q)(p)[q]),
$$
for all $p,q$.  
 
The natural spectral sequence 
$$
E_2^{ij}=\Ext^{i}_{MHS}(\Q,H^{j}(X,\Q)(p))\Rightarrow H^{i+j}_{\mc{H}}(X,\Q(p)),
$$
and vanishing of $\Ext^i$ for $i>1$, induces short exact sequences
\begin{equation}\label{ssabsolute Hodge}
0\xr{} \Ext^1(\Q,H^{q-1}(X,\Q)(p)) \xr{} H^q_{\mc{H}}(X,\Q(p)) \xr{} 
\Hom(\Q,H^q(X,\Q)(p)) \xr{} 0.
\end{equation}
Note that $\Hom$s and 
$\Ext$s are taken in the category of mixed Hodge structures. 

If $X$ is smooth and proper then we have a comparison isomorphism 
with Deligne cohomology 
$$
H^q_{\mc{H}}(X,\Q(p)) \cong H^q_{\mc{D}}(X,\Q(p)),
$$
provided that $q\leq 2p$ \cite[\textsection2.7]{JA}.

\subsection{Cycle class map} 

Let $DM_{gm,\Q}$ be Voevodsky's triangulated category of motivic complexes
with rational coefficients over $\C$ (\cite{Vt},\cite{V}). Denoting by 
$Sm/\C$ the category of smooth complex algebraic varieties, there is a functor 
$$
Sm/\C \xr{} DM_{gm,\Q}, \quad X\mapsto M_{gm}(X).
$$
Motivic cohomology is defined by 
$$
H_{M}^{q}(X,\Q(p))=\Hom_{DM_{gm}}(M_{gm}(X),\Q(p)[q]),
$$
for $X$ smooth and $p\geq 0, q\in \Z$. There is a comparison isomorphism \cite{Vh} with Bloch's higher Chow groups  
$$
H_{M}^{q}(X,\Q(p))\cong \CH^{p}(X,2p-q)\otimes \Q.
$$ 
By Levine \cite{Le} and Huber (\cite{Hu1},\cite{Hu2}) we have realizations 
\begin{equation}\label{equation-realization}
r_{\mc{H}}:DM_{gm}\xr{} D^b(MHS)
\end{equation}
at disposal, such that $\underline{R\Gamma}(X,\Q)$ is the dual of $r_{\mc{H}}(M_{gm}(X))$ and  
$r_{\mc{H}}(\Q(1))=\Q(1)$. 
The realizations are triangulated $\otimes$-functors and therefore induce
cycle maps
\begin{equation}\label{cycle-map}
c_{p,2p-q}: H^{q}_M(X,\Q(p))\xr{} H^{q}_{\mc{H}}(X,\Q(p))
\end{equation}
which are compatible with the localization sequence. An explicit construction
of a cycle map by using currents is presented in \cite{K-L-abeljacobi}.  
Composition of $c_{p,2p-q}$ with the projection from \eqref{ssabsolute Hodge} yields
\begin{equation}\label{bar-cycle-map}
\bar{c}_{p,2p-q}: H^{q}_M(X,\Q(p))\xr{} \Hom_{MHS}(\Q,H^{q}(X,\Q)(p)).
\end{equation}

\subsection{Beilinson-Hodge conjecture} 
Let $X$ be a smooth algebraic variety over $\C$, and $n\geq 0$ an integer.
\begin{conjecture}[Beilinson-Hodge conjecture]
BH(X,n): The cycle map $\bar{c}_{n,n}$ \eqref{bar-cycle-map}  is surjective.  
\end{conjecture}

\begin{remark}\label{remark-n=1}
If $X$ is smooth then 
$$ 
c_{1,1}:H_M^1(X,\Q(1))\xr{} H_{\mc{H}}^1(X,\Q(1))
$$ 
is an isomorphism  \cite[Proposition~2.12]{EV}. In particular, BH(X,1) holds.
\end{remark}
 
\section{Beilinson-Hodge conjectures for the generic point}

\subsection{Coniveau spectral sequences}
The main technical tool of our paper is the coniveau spectral sequence 
for motivic and absolute Hodge cohomology. The existence and construction 
of the coniveau spectral is well-known and follows from the yoga of exact couples
as in \cite{BO}. Because we couldn't provide a reference for the case
of absolute Hodge cohomology we will explain the construction 
in this section.

\subsubsection{} In the following, $?$ will stand for $M$ or $\mc{H}$.
For $p\geq 0$ we denote by $X^{(p)}$ the set of codimension $p$ points of $X$. 
For a point $x\in X$ (not necessarily closed) we define
$$
H^{q}_{?}(x,\Q(p)):=\varinjlim_{x\in U\subset X} H^q_{?}(U,\Q(p)),
$$
where $U$ runs over all open neighborhoods of $x$. For all $n\geq 0$, the coniveau 
spectral sequence reads:
\begin{equation}\label{equation-coniveau-spectral-seq}
E_{1}^{p,q}=\bigoplus_{x\in X^{(p)}, p\leq n} H^{q-p}_{?}(x,\Q(n-p))\Rightarrow H^{p+q}_?(X,\Q(n)).
\end{equation}
The terms $E_1^{p,q}$ with $p>n$ are zero.

\subsubsection{} 
In order to construct the coniveau spectral sequence we use the 
category of finite correspondences $Cor_\C$ \cite[Lecture~1]{V}.
There is an obvious functor 
$$
Sm/\C\xr{} Cor_\C, \quad X\mapsto [X].
$$ 
We denote by $\mc{H}^b(Cor_\C)$ the homotopy category of bounded complexes \cite[\textsection2.1]{Vt}. 
By construction \cite[Definition~2.1.1]{Vt} there is a triangulated 
functor
\begin{equation}\label{equation-functor-Cor-to-DM}
\mc{H}^b(Cor_\C) \xr{} DM_{gm,\Q}.
\end{equation}

\begin{definition}
Let $X$ be smooth and $Y\subset X$ a closed set. We define $c_{Y}X$
to be the complex
$$
[X\backslash Y]\xr{\jmath} \underset{\deg=0}{[X]}
$$ 
in $\mc{H}^b(Cor_\C)$. The map $\jmath$ is the open immersion.
\end{definition}

Let $Y\subset X$ be a closed subset. For an open subset $U$ of  $X$
and a closed subset $Y'$ of  $U$ such that $Y\cap U\subset Y'$
we get a morphism of complexes
\begin{equation}\label{equation-c-maps}
c_{Y'}U\xr{} c_{Y}X.
\end{equation}

\begin{lemma}\label{lemma-distingueshed-triangles}
Let $X$ be smooth.
Let $Y_1,Y_2$ be closed subsets of $X$ with $Y_2\subset Y_1$. 
\begin{enumerate}
\item 
The morphisms 
$$
c_{Y_1\backslash Y_2}(X\backslash Y_2) \xr{} c_{Y_1}X \xr{} c_{Y_2}X \xr{+1} c_{Y_1\backslash Y_2}(X\backslash Y_2) [1]
$$
induced by \eqref{equation-c-maps} and 
\begin{align*}
c_{Y_2}X\xr{} c_{Y_1\backslash Y_2}(X\backslash Y_2)[1] \\
\xymatrix
{
[X] 
&
\\
[X\backslash Y_2]\ar[u]\ar[r]^{-{\rm id}}
&
[X\backslash Y_2]
\\
&
[X\backslash Y_1]\ar[u]^{-{\rm incl}}
}
\end{align*}
form a distinguished triangle in $\mc{H}^b(Cor_\C)$.
\item If $Y_2'\subset Y'_1$ are closed subsets of $X$ such that $Y_i\subset Y'_i$, for $i=1,2$, then the morphisms from \eqref{equation-c-maps} induce
a morphism of distinguished triangles 
$$
\xymatrix 
{
c_{Y_1\backslash Y_2}(X\backslash Y_2) \ar[r] & 
c_{Y_1}X \ar[r] & 
c_{Y_2}X \ar[r]^{+1}&
\\
c_{Y'_1\backslash Y'_2}(X\backslash Y'_2) \ar[r]\ar[u] & 
c_{Y'_1}X \ar[r]\ar[u] & 
c_{Y'_2}X \ar[r]^{+1}\ar[u]
&
}
$$ 
\end{enumerate}
\begin{proof}
For (1). 
equivalent to $0$. There is an obvious isomorphism in $\mc{H}^b(Cor_\C):$
\begin{align*}
{\rm cone}(c_{Y_1}X \xr{} c_{Y_2}X)[-1]\xr{} c_{Y_1\backslash Y_2}(X\backslash Y_2),\\
\xymatrix
{
[X]
&
&
\\
[X] \oplus [X\backslash Y_2] \ar[u]\ar[rr]^{-{\rm pr}_2}
&
&
[X\backslash Y_2]
\\
[X\backslash Y_1]\ar[u]^{({\rm incl},-{\rm incl})}\ar[rr]^{=}
&
&
[X\backslash Y_1],\ar[u]
}
\end{align*}
rendering commutative the diagram 
$$
\xymatrix 
{
c_{Y_2}X[-1]\ar[r] \ar[d]&
{\rm cone}(c_{Y_1}X \xr{} c_{Y_2}X)[-1]\ar[r] \ar[d]&
c_{Y_1}X \ar[r] \ar[d]& 
c_{Y_2}X  \ar[d]
\\
c_{Y_2}X[-1]\ar[r]&
c_{Y_1\backslash Y_2}(X\backslash Y_2) \ar[r] & 
c_{Y_1}X \ar[r] & 
c_{Y_2}X. 
}
$$

For (2). Straight-forward.
\end{proof}
\end{lemma}

\begin{definition}
Let $X$ be smooth and $Y\subset X$ a closed subset. 
For all $n\geq 0$ and $q\in \Z$ we define 
\begin{align*}
H^q_{Y,M}(X,\Q(n))&:= \Hom_{DM_{gm}}(c_YX,\Q(n)[q])\\
H^q_{Y,\mc{H}}(X,\Q(n))&:= \Hom_{D^b(MHS)}(r_{\mc{H}}(c_YX),\Q(n)[q]).
\end{align*}
We implicitly used the functor \eqref{equation-functor-Cor-to-DM}.
\end{definition}

From \eqref{equation-c-maps} we obtain a map
$$
H^*_{Y,?}(X,\Q(n)) \xr{}  H^*_{Y',?}(U,\Q(n))
$$
if  $U$ is an open subset of $X$ and $Y\cap U\subset Y'$.

\subsubsection{} 
For $p\geq 0$ we denote by $Z^p=Z^p(X)$ the set  closed 
subsets of $X$ of codimension $\geq p$, ordered by inclusion. Let $Z^p/Z^{p+1}$ denote the ordered set of pairs 
$(Z,Z')\in Z^p\times Z^{p+1}$ such that $Z\supset Z'$, with the ordering 
$$
(Z,Z')\geq (Z_1,Z'_1) \quad \text{if $Z\supset Z_1$ and $Z'\supset Z'_1$.}
$$
We can form for all $n\geq 0$ and $p\in \Z$:
\begin{align*}
H^*_{Z^p,?}(X,\Q(n))&:=\begin{cases} \varinjlim_{Z\in Z^p} H^*_{Z,?}(X,\Q(n)) &\text{if $p\geq 0$,}\\
H^*_?(X,\Q(n)) &\text{if $p\leq 0$.}\end{cases}\\
H^*_{Z^p/Z^{p+1},?}(X,\Q(n))&:=\begin{cases} \varinjlim_{(Z,Z')\in Z^p/Z^{p+1}} 
H^*_{Z\backslash Z',?}(X\backslash Z',\Q(n)) &\text{if $p\geq 0$,}\\
0 &\text{if $p<0$.}\end{cases}
\end{align*}

In view of Lemma \ref{lemma-distingueshed-triangles}, we obtain for every 
$(Z,Z')\in Z^p/Z^{p+1}$ a long exact sequence
\begin{equation}\label{equation-long-exact-sequence-cohomology-supports}
H^*_{Z',?}(X,\Q(n))\xr{} H^*_{Z,?}(X,\Q(n)) \xr{} 
H^*_{Z\backslash Z',?}(X\backslash Z',\Q(n))\xr{+1}, 
\end{equation}
and we can take the limit to get a long exact sequence
\begin{equation}\label{equation-triangle}
H^*_{Z^{p+1},?}(X,\Q(n))\xr{} H^*_{Z^p,?}(X,\Q(n)) \xr{} 
H^*_{Z^p/Z^{p+1},?}(X,\Q(n))\xr{+1}. 
\end{equation}
This also holds for $p<0$ for trivial reasons. We form an exact couple as 
follows
\begin{align}
D&:=\bigoplus_{p\in \Z} H^*_{Z^p,?}(X,\Q(n)), \label{align-exact-couple} \\
E&:= \bigoplus_{p\geq 0} H^*_{Z^p/Z^{p+1},?}(X,\Q(n)), \nonumber
\end{align}
and the exact triangle induced by \eqref{equation-triangle}:
$$
\xymatrix
{
D \ar[rr]
&
&
D \ar[dl]
\\
&
E. \ar[ul]
&
}
$$
Setting 
\begin{equation*}
E^{p,q}_{1}:=H^{p+q}_{Z^p/Z^{p+1},?}(X,\Q(n)), 
\end{equation*}
the exact couple yields a spectral sequence
\begin{equation}\label{equation-coniveau-spectral-seq-support}
E^{p,q}_{1}\Rightarrow H^{p+q}_?(X,\Q(n)),
\end{equation}
for all $n\geq 0$, such that 
$$
E^{p,q}_{\infty}= \frac{N^pH^{p+q}_?(X,\Q(n))}{N^{p+1}H^{p+q}_?(X,\Q(n))},
$$
with 
$$
N^{p}H^{i}_?(X,\Q(n))={\rm image}(H^i_{Z^p,?}(X,\Q(n)) \xr{} H^i_{?}(X,\Q(n))).
$$

\begin{lemma} \label{lemma-purity}
Let $X$ be smooth and $n\geq 0$.
\begin{enumerate}
\item If $p\leq n$ then 
$$
H^{q+p}_{Z^p/Z^{p+1},?}(X,\Q(n))\cong \bigoplus_{x\in X^{(p)}} H^{q-p}(x,\Q(n-p)) , \quad \text{for all $q\in \Z$.}
$$
\item If $p>n$ then 
$$
H^{q}_{Z^p/Z^{p+1},?}(X,\Q(n))=0,\quad \text{for all $q\in \Z$.}
$$
\end{enumerate}
\begin{proof}
The set 
$$
S=\{(Z,Z')\in Z^p/Z^{p+1}\mid \text{$Z\backslash Z'$ is smooth of pure codimension $=p$}\}$$ 
is a cofinal subset of $Z^p/Z^{p+1}$; thus 
$$
H^{q+p}_{Z^p/Z^{p+1},?}(X,\Q(n))=\varinjlim_{(Z,Z')\in S} H^*_{Z\backslash Z',?}(X\backslash Z',\Q(n))
$$
From the Gysin triangle \cite[Proposition~3.5.4]{Vt} we obtain for all 
$(Z,Z')\in S$ a natural isomorphism 
$$
c_{Z\backslash Z'} (X\backslash Z') \cong M_{gm}(Z\backslash Z')(p)[2p],
$$
in $DM_{gm}$.

For (1). By using cancellation we obtain 
$$
H^{q+p}_{Z\backslash Z',?}(X\backslash Z',\Q(n))=H^{q-p}_{?}(Z\backslash Z',\Q(n-p))
$$
for all $(Z,Z')\in S$, and the restriction maps 
$$
H^{q-p}_{?}(Z\backslash Z',\Q(n-p)) \xr{} \bigoplus_{x\in X^{(p)}} H^{q-p}_?(x,\Q(n-p))
$$
induce the desired isomorphism.

For (2). Suppose $p>n$. We claim that 
\begin{equation}\label{equation-H*Z}
H^{*}_{Z,?}(X,\Q(n))=0 
\end{equation}
for all $Z\in Z^p(X)$. In view of the long exact sequence 
\eqref{equation-long-exact-sequence-cohomology-supports} this will prove the claim. 

By definition the vanishing of $H^{*}_{Z,?}(X,\Q(n))$ follows if the restriction map
$$
H^*_?(X,\Q(n))\xr{} H^*_?(X\backslash Z,\Q(n))
$$
is an isomorphism. Set $U:=X\backslash Z$. For $?=M$ we can use 
the comparison isomorphism with higher Chow groups. It is
sufficient to prove that the restriction induces an isomorphism 
of complexes 
\begin{equation}\label{equation-restriction-X-to-U}
Z^n(X,\bullet)\xr{\cong} Z^n(U,\bullet),
\end{equation}
where $Z^n(?,\bullet)$ denotes Bloch's cycle complex.
Since $X\backslash U$ has codimension $>n$, the map \eqref{equation-restriction-X-to-U}
is injective. For the surjectivity, let $A\in Z^n(U,m)$ be the class of an 
irreducible subvariety of $U\times \Delta^m$. By definition $A$ has codimension
$n$ and meets all faces $U\times \Delta^{i}$ properly. Let $\bar{A}$ be 
the closure of $A$ in $X\times \Delta^{m}$. Since 
$$
\bar{A}\cap (X\times \Delta^{i})\subset (A\cap (U\times \Delta^{i}))\cup ((X\backslash U)\times \Delta^{i}),
$$
and $(X\backslash U)\times \Delta^{i}$ has codimension $>n$ in 
$X\times \Delta^{i}$, we conclude that $\bar{A}\in Z^n(X,m)$.

For $?=\mc{H}$. 
In view of \eqref{ssabsolute Hodge}, we need to prove that the restriction induces 
isomorphisms 
\begin{align}
\label{align-Hom} \Hom_{MHS}(\Q(-n),H^q(X,\Q)) &\xr{\cong} \Hom_{MHS}(\Q(-n),H^q(U,\Q)), \\
\label{align-Ext} \Ext^1_{MHS}(\Q(-n),H^q(X,\Q)) &\xr{\cong} \Ext^1_{MHS}(\Q(-n),H^q(U,\Q)),
\end{align}
for all $q$. In order to prove \eqref{align-Hom} and \eqref{align-Ext} we 
use the exact sequence
\begin{multline}\label{multline-exact-seq-HS}
0\xr{} \frac{H^q_{X\backslash U}(X,\Q)}{{\rm  im}(H^{q-1}(U,\Q))} \xr{} H^q(X,\Q) \xr{} H^q(U,\Q)\xr{} \\
\ker\left(H^{q+1}_{X\backslash U}(X,\Q)\xr{} H^{q+1}(X,\Q)\right)\xr{} 0.
\end{multline}
Note that $\frac{H^q_{X\backslash U}(X,\Q)}{{\rm  im}(H^{q-1}(U,\Q))}$ 
and $\ker(H^{q+1}_{X\backslash U}(X,\Q)\xr{} H^{q+1}(X,\Q))$ are Hodge structures of weight $\geq 2p$. If $E$ is any mixed Hodge structure of weight $\geq 2p$ then 
$$
 \Hom(\Q(-n),E)=0, \quad \Ext^1(\Q(-n),E)=0,
$$
because $p>n$. Therefore \eqref{multline-exact-seq-HS} implies the statement. 
\end{proof}
\end{lemma}

\begin{proposition} \label{proposition-coniveau-spectral-seq}
Let $X$ be smooth and $?=M$ or $?=\mc{H}$. Let $n\geq 0$ be an integer.
\begin{enumerate}
\item  There is a  spectral sequence 
$$
E_{1,?}^{p,q}=\bigoplus_{x\in X^{(p)}, p\leq n} H^{q-p}_{?}(x,\Q(n-p))\Rightarrow H_?^{p+q}(X,\Q(n))
$$
such that 
$$
E^{p,q}_{\infty,?}= \frac{N^pH^{p+q}_?(X,\Q(n))}{N^{p+1}H^{p+q}_?(X,\Q(n))},
$$
with 
$$
N^{p}H^{*}_?(X,\Q(n))=\bigcup_{\substack{U\subset X \\ {\rm cd}(X\backslash U)\geq p}} 
\ker(H^*_?(X,\Q(n)) \xr{} H^*_{?}(U,\Q(n))),
$$
where $U$ runs over all open subsets with $\codim (X\backslash U)\geq p$.
\item The cycle map induces a morphism of spectral sequences 
$$
[E_{1,M}^{p,q}\Rightarrow H_M^{p+q}(X,\Q(n))] \xr{} [E_{1,\mc{H}}^{p,q}\Rightarrow H_{\mc{H}}^{p+q}(X,\Q(n))].
$$
\end{enumerate}
\begin{proof}
For (1). The statement  follows from the spectral sequence \eqref{equation-coniveau-spectral-seq-support}  and Lemma \ref{lemma-purity}.

For (2). The realization $r_{\mc{H}}$ \eqref{equation-realization} induces
a morphism of the exact couples \eqref{align-exact-couple}. 
\end{proof}
\end{proposition}

\subsection{$E_1$ complexes of the coniveau spectral sequence}
Let $X$ be smooth and connected, we denote by $\eta$ the generic point of $X$. 
The cycle map induces a morphism between the 
$E^{\bullet,2}_1$ complexes of the coniveau spectral sequence 
(Proposition \ref{proposition-coniveau-spectral-seq}) for $n=2$:  
\begin{equation}\label{diagram-Gersten}
\xymatrix
{
E^{\bullet,2}_{1,M}:
H^2_M(\eta,\Q(2)) \ar[r]\ar[d]
&
\oplus_{x\in X^{(1)}} H^1_M(x,\Q(1))\ar[r]\ar[d]
&
\oplus_{x\in X^{(2)}} \Q \ar[d]
\\
E^{\bullet,2}_{1,\mc{H}}:
H^2_{\mc{H}}(\eta,\Q(2)) \ar[r]\ar[d]
&
\oplus_{x\in X^{(1)}} H^1_{\mc{H}}(x,\Q(1))\ar[r]\ar[d]
&
\oplus_{x\in X^{(2)}} \Q \ar[d]
\\
\Hom(\Q,H^2(\eta,\Q)(2)) \ar[r]
&
\oplus_{x\in X^{(1)}} \Hom(\Q,H^1(x,\Q)(1))\ar[r]
&
\oplus_{x\in X^{(2)}} \Q 
}
\end{equation}
We call the complex in the first line $G_{M}(X,2)$, 
the complex in the second line $G_{\mc{H}}(X,2)$, and finally the complex in the 
third line is called $G_{HS}(X,2)$. The complex $G_{HS}(X,2)$ is induced
by $G_{\mc{H}}(X,2)$ via \eqref{ssabsolute Hodge}.

For $G_M(X,2)$ the group $H^2_M(\eta,\Q(2))$
is the component in degree $=0$, and the grading is defined similarly for
$G_{\mc{H}}(X,2)$ and $G_{HS}(X,2)$. Via Gersten-Quillen resolution we have 
\begin{equation}\label{equation-H1K2}
H^1(G_M(X,2))=H^1(X,\mc{K}_2)\otimes_{\Z} \Q,
\end{equation} 
where  $\mc{K}_2$ is Quillen's K-theory  Zariski sheaf associated to the presheaf 
$U\mapsto K_2(\OO_X(U))$. 

\begin{proposition}\label{proposition-cohomology-Gersten}
Let $X$ be smooth and connected.
\begin{itemize}
\item[(i)] There is a natural isomorphism $H^1(G_M(X,2))\xr{\cong} H^3_M(X,\Q(2)).$
\item[(ii)] There is a natural injective map 
$
H^1(G_{\mc{H}}(X,2))\xr{} H^3_{\mc{H}}(X,\Q(2)).
$
We call the image $H^3_{{\mc{H}},{\rm alg}}(X,\Q(2))$.
\item[(iii)] There is a natural  isomorphism 
$$
H^1(G_{HS}(X,2))\xr{} H^3_{\mc{H},{\rm alg}}(X,\Q(2))/\left(H^1_{\mc{H}}(\C,\Q(1))\cdot H^2_{\mc{H}}(X,\Q(1))\right).
$$
\item[(iv)] The above maps form a commutative diagram 
$$
\xymatrix
{
H^1(G_M(X,2))\ar[r]^{\cong} \ar[d] 
&
H^3_M(X,\Q(2))\are[d]^{c_{2,1}}
\\
H^1(G_{\mc{H}}(X,2))\ar[r]^{\cong} \ar[d] 
&
H^3_{{\mc{H}},{\rm alg}}(X,\Q(2))\ar[d] 
\\
H^1(G_{HS}(X,2))\ar[r]^-{\cong} 
&
H^3_{{\mc{H}},{\rm alg}}(X,\Q(2))/H^1_{{\mc{H}}}(\C,\Q(1))\cdot H^2_{{\mc{H}}}(X,\Q(1)),  
}
$$
and 
$$
c_{2,1}:H^3_M(X,\Q(2))\xr{} H^3_{{\mc{H}},{\rm alg}}(X,\Q(2))
$$
is surjective.
\end{itemize}
\end{proposition}

\begin{proof}
Statement (i) is proved in \cite{MS}.

\emph{Proof of (i) and (ii).} We use the coniveau spectral sequence 
(Proposition \ref{proposition-coniveau-spectral-seq})  
\begin{equation}
\label{equation-spectral-sequence}
E_{1}^{p,q}=\bigoplus_{x\in X^{(p)}} H^{q-p}_{?}(x,\Q(n-p)) \Rightarrow H^{p+q}_{?}(X,\Q(n)),
\end{equation}
where $?$ is $M$ or $\mc{H}$, and $0\leq p\leq n$. We have 
$$
G_{?}(X,2)=E^{\bullet,2}_1
$$
for $n=2$. 
We get $E_{2}^{1,2}=E_{\infty}^{1,2}$ and $E_{\infty}^{2,1}=0=E^{3,0}_{\infty}$
for obvious reasons. Therefore we obtain an exact sequence
\begin{equation}\label{equation-short-exact-sequence-H32}
0\xr{} E^{1,2}_{\infty}\xr{} H^3_{?}(X,\Q(2))\xr{} E^{0,3}_{\infty} \xr{} 0,
\end{equation}
with  
$$
E^{1,2}_{\infty}=\ker(H^3_{?}(X,\Q(2))\xr{} H^3_{?}(\eta,\Q(2))).
$$
For $?=M$ we have $H^3_{M}(\eta,\Q(2))=0$; for $?=\mc{H}$ we define 
$$
H^3_{{\mc{H}},{\rm alg}}(X,\Q(2)):=\ker(H^3_{\mc{H}}(X,\Q(2))\xr{} H^3_{\mc{H}}(\eta,\Q(2)))=N^1H^3_{\mc{H}}(X,\Q(2)).
$$

For (iii). 
If $X$ is smooth then it is not difficult to see that
$$
{\rm Pic}(X)\otimes \Q = H^2_M(X,\Q(1)) \cong H^2_{\mc{H}}(X,\Q(1)).
$$
It follows that via the isomorphism 
$H^1(G_{\mc{H}}(X,2))\cong H^3_{\mc{H},{\rm alg}}(X,\Q(2))$ the subgroup 
$H^1_{\mc{H}}(\C,\Q(1))\cdot H^2_{\mc{H}}(X,\Q(1))$ of $H^3_{\mc{H},{\rm alg}}(X,\Q(2))$ corresponds to the image
of $\oplus_{x\in X^{(1)}}\C^*\otimes_{\Z}\Q$ in $H^1(G_{\mc{H}}(X,2))$. 
For every point $x\in X^{(1)}$ we have an exact sequence 
$$
0\xr{} \C^*\otimes_{\Z} \Q \xr{} H^1_{\mc{H}}(x,\Q(1))\xr{} \Hom_{MHS}(\Q,H^1(x,\Q)(1))\xr{} 0,
$$
and therefore 
$$
\ker(H^1(G_{\mc{H}}(X,2))\xr{} H^1(G_{HS}(X,2)))={\rm im}(\oplus_{x\in X^{(1)}}\C^*\otimes_{\Z}\Q \xr{} H^1(G_{\mc{H}}(X,2))).
$$
This implies the claim.

Statement (iv) is obvious.
\end{proof}

\begin{remark}
For a smooth projective variety $X$ we know that 
$$
H^1(G_{HS}(X,2))\xr{\cong} H^3_{{\mc{H}},{\rm alg}}(X,\Q(2))/\left(H^1_{\mc{H}}(\C,\Q(1))\cdot H^2_{\mc{H}}(X,\Q(1))\right)
$$
is a countable group \cite{MS}.  
This is a consequence of the fact that 
deformations $a'$ of a class $a\in H^3_M(X,\Q(2))$ have the same image via 
$c_{2,1}$ modulo the group $H^1_{\mc{H}}(\C,\Q(1))\cdot H^2_{\mc{H}}(X,\Q(1))$.
There exist examples of K3-surfaces $X$ such that $H^1(G_{HS}(X,2))\neq 0$ \cite{MS}. 
\end{remark}

\begin{definition}\label{definition-decomposable}
Let $X$ be smooth, connected and projective. 
We denote by 
$$
{\rm image}(H^1_M(\C,\Q(1))\cdot H^2_M(X,\Q(1)))=:H^3_M(X,\Q(2))_{{\rm dec}}\subset H^3_M(X,\Q(2)) 
$$ 
the subgroup of decomposable cycles. In the same way we define $H^3_{\mc{H}}(X,\Q(2))_{{\rm dec}}$. 
\end{definition}

Note that 
\begin{align} \label{align-shouldbealemma1}
\C^*\otimes_{\Z}\Q&=H^1_M(\C,\Q(1))\cong H^1_{\mc{H}}(\C,\Q(1)), \\ 
{\rm Pic(X)}\otimes_{\Z}\Q&=H^2_M(X,\Q(1)))\cong H^2_{\mc{H}}(X,\Q(1))). \label{align-shouldbealemma2}
\end{align}

\begin{lemma}\label{lemma-injective-dec}
If $X$ is smooth, projective, and $H^1(X)=0$, then the maps 
\begin{align*}
H^1_{M}(\C,\Q(1))\otimes_{\Q} H^2_{M}(X,\Q(1)))&\xr{} H^3_M(X,\Q(2))\\
H^1_{\mc{H}}(\C,\Q(1))\otimes_{\Q} H^2_{\mc{H}}(X,\Q(1)))&\xr{} H^3_{\mc{H}}(X,\Q(2))
\end{align*}
are injective. In particular, 
$$
H^3_M(X,\Q(2))_{{\rm dec}} \xr{} H^3_{\mc{H}}(X,\Q(2))_{{\rm dec}}
$$
is an isomorphism.
\end{lemma}

\begin{proof}
By using the cycle map it is sufficient to prove the statement 
for absolute Hodge cohomology. The assumption $H^1(X)=0$ implies 
$$
H^2_{\mc{H}}(X,\Q(1))) \cong  \Hom(\Q(-1),H^2(X,\Q)).
$$ 
The pure Hodge structure $H^2(X,\Q)$ is polarizable and therefore  
$$
\Hom(\Q(-1),H^2(X,\Q))\otimes \Q(-1)\subset H^2(X,\Q)
$$
is a direct summand which we call  ${\rm Hg}^{1,1}$. We get 
\begin{multline*}
({\rm Hg}^{1,1}\otimes_{\Q} \C)/(2\pi i)^2\cdot {\rm Hg}_{\Q}^{1,1}=\Ext^1(\Q(-2),{\rm Hg}^{1,1}) \\ \subset \Ext^1(\Q(-2),H^2(X,\Q)) \subset H^3_{\mc{H}}(X,\Q(2))  
\end{multline*}
and clearly $H^1_{\mc{H}}(\C,\Q(1))\otimes_{\Q} H^2_{\mc{H}}(X,\Q(1)))$ is
mapping isomorphically onto 
$
({\rm Hg}^{1,1}\otimes_{\Q} \C)/(2\pi i)^2\cdot {\rm Hg}_{\Q}^{1,1}.
$ 
Because of \eqref{align-shouldbealemma1} and \eqref{align-shouldbealemma2} we 
conclude that 
$$
H^3_M(X,\Q(2))_{{\rm dec}} \xr{} H^3_{\mc{H}}(X,\Q(2))_{{\rm dec}}
$$
is an isomorphism.
\end{proof}

\begin{proposition}\label{proposition-Gersten-H0}
Let $X$ be smooth and connected. Restriction to the generic point 
yields the following equalities:
\begin{align}
H^2_M(X,\Q(2))&\xr{\cong} H^0(G_{M}(X,2)) \label{align-1-Gersten-H0} \\
H^2_{\mc{H}}(X,\Q(2))&\xr{\cong} H^0(G_{\mc{H}}(X,2)). \label{align-2-Gersten-H0} 
\end{align}
\begin{proof}
We use the coniveau spectral sequence (Proposition \ref{proposition-coniveau-spectral-seq}) for $n=2$. We have
$$
E^{0,2}_{2,?}=H^0(G_?(X,2))
$$
for $?=M$ and $?=\mc{H}$. Note that $E^{2,q}_{1,?}=0$ for $q\neq 2$. Thus 
$E^{0,2}_{2,?}=E^{0,2}_{\infty,?}$ and $E^{2,0}_{\infty,?}=0$. Moreover, 
$E^{1,1}_{1,?}=0$, because $H^0_?(U,\Q(1))=0$ for every $U$. It follows
that $E^{1,1}_{\infty,?}=0$ and 
$$
H^2_{?}(X,\Q(2))=E^{0,2}_{\infty,?}=E^{0,2}_{2,?}.
$$
\end{proof}
\end{proposition}

\begin{lemma}\label{lemma-stupid-reduction-genpoint}
Let $X$ be smooth and connected. We denote by $\eta$ the generic point of $X$. 
If 
$$
H^2_M(\eta,\Q(2))\xr{} H^2_{\mc{H}}(\eta,\Q(2))
$$
is surjective then
$$
H^2_M(X,\Q(2))\xr{} H^2_{\mc{H}}(X,\Q(2))
$$
is surjective.
\begin{proof}
We use Proposition \ref{proposition-Gersten-H0} and need to prove that 
$$
H^0(G_{M}(X,2))\xr{} H^0(G_{\mc{H}}(X,2))
$$  
is surjective.
Since $H^1_M(x,\Q(1))=H^1_{\mc{H}}(x,\Q(1))$ for every point $x\in X$ of codimension 
$=1$ this follows immediately from diagram \eqref{diagram-Gersten}.
\end{proof}
\end{lemma}

\begin{proposition}\label{proposition-HS-Gersten-H0}
If $X$ is smooth, connected and projective then 
\begin{equation*}
H^0(G_{HS}(X,2))=0. 
\end{equation*}
\begin{proof}
Let $\eta\in X$ be the generic point.
For 
$$a\in H^0(G_{HS}(X,2))\subset \Hom(\Q(-2),H^2(\eta,\Q))$$ 
we can find an
effective divisor $D$ such that $a$ is induced by a cohomology class 
$a'\in \Hom(\Q(-2),H^2(X\backslash D,\Q))$. Let $S$ be a closed subset of 
$X$ of codimension $\geq 2$, such that $D\backslash S$ is smooth. Denoting 
$X':=X\backslash S$, we claim that $a'\mid_{X'\backslash D}$ maps to zero 
in $H^1(D\backslash S,\Q)(-1)$ via the boundary map of the localization 
sequence for singular cohomology. Indeed, the map 
$$
\Hom(\Q(-1), H^1(D\backslash S,\Q))\xr{} \bigoplus_{x\in X^{(1)}}\Hom(\Q(-1), H^1(x,\Q))
$$
is injective and therefore the claim follows from $a\in H^0(G_{HS}(X,2))$.

Now $a'\mid_{X'\backslash D}$ defines an extension of Hodge structures 
$$
0\xr{} {\rm image}(\Q(-1)^{\pi_0(D\backslash S)})\xr{} E \xr{} \Q(-2)\xr{} 0,
$$   
with $E\subset H^2(X',\Q)$.
We note that $H^2(X',\Q)=H^2(X,\Q)$ is a pure Hodge structure of weight $=2$,
and therefore the same holds for $E$. Thus the extension is trivial and 
$a'\mid_{X'\backslash D}$ lifts to $\Hom(\Q(-2),H^2(X',\Q))=0$. 
This proves that $a'\mid_{X'\backslash D}=0$ and implies $a=0$.
\end{proof}
\end{proposition}

\subsection{An exact sequence for projective varieties with vanishing $H^1$}

\begin{lemma}\label{lemma-reglobalisation}
Let $X$ be smooth, projective and connected. 
Suppose that $H^1(X)=0$.
We denote by $\eta$ the 
generic point of $X$. 
There is an exact sequence
$$
H^2_M(\eta,\Q(2))\xr{}  H^2_{\mc{H}}(\eta,\Q(2)) \xr{} H^3_M(X,\Q(2)) \xr{} H^3_{{\mc{H}}}(X,\Q(2)).
$$
\begin{proof}
Via Proposition \ref{proposition-cohomology-Gersten} we identify 
$$
H^3_M(X,\Q(2))\cong H^1(G_M(X,2)), \quad H^1(G_{\mc{H}}(X,2))\subset H^3_{{\mc{H}}}(X,\Q(2))
$$
and need to show that there is an exact sequence 
$$
H^2_M(\eta,\Q(2))\xr{}  H^2_{\mc{H}}(\eta,\Q(2)) \xr{} H^1(G_M(X,2)) \xr{} H^1(G_{\mc{H}}(X,2)).
$$
We work with diagram \eqref{diagram-Gersten}. The map 
$$
H^2_{\mc{H}}(\eta,\Q(2))\xr{} H^1(G_M(X,2))
$$
is defined by using $E^{1,2}_{1,M}=E^{1,2}_{1,\mc{H}}$.
The assumptions on $X$ 
imply that $H^2_{\mc{H}}(X,\Q(2))=0$ and therefore $H^0(G_{\mc{H}}(X,2))=0$ by
Proposition \ref{proposition-Gersten-H0}. The rest of the proof
involves only diagram chasing.
\end{proof}
\end{lemma}

\subsection{Main theorem}

\begin{thm}\label{main-thm}
Let $X$ be smooth and connected. Let $\bar{X}$ be a smooth compactification of $X$.
We denote by $\CH_0(\bar{X})\otimes_{\Z}\Q$ the Chow group of zero cycles on $\bar{X}$. 
If $\deg:\CH_0(\bar{X})\otimes_{\Z}\Q \xr{} \Q$ is an isomorphism then $BH(X,2)$ holds.
\end{thm}





\begin{proof}
In view of Lemma \ref{lemma-stupid-reduction-genpoint} 
it is sufficient to show that 
$
\bar{c}_{2,2}:H^2_M(\eta,\Q(2))\xr{} H^2_{\mc{H}}(\eta,\Q(2))
$
is surjective, where $\eta$ is the generic point on $X$. 

Now, choose a smooth projective
model $Y$ of $\eta$. Since $\bar{X}$ and $Y$ are birational we conclude that
$\CH_0(Y)\otimes_{\Z}\Q\cong \Q$. It follows that $H^1(Y)=0$.

We claim that
$$
H^3_M(Y,\Q(2))_{{\rm dec}} = H^3_M(Y,\Q(2)).
$$
This implies the theorem by using Lemma \ref{lemma-reglobalisation}, because
$$
H^3_M(Y,\Q(2))_{{\rm dec}} \subset H^3_{\mc{H}}(Y,\Q(2))_{{\rm dec}}
$$
by Lemma \ref{lemma-injective-dec}.

In view of Proposition \ref{proposition-cohomology-Gersten} and \eqref{equation-H1K2}
it is sufficient to prove that the cokernel of 
$$ 
\C^*\otimes_{\Z} {\rm Pic(X)}\xr{} H^1(X,\mc{K}_2)
$$
is torsion. This is proved in \cite[Theorem~3(i)]{BS}.
\end{proof} 

\subsection{An exact sequence for projective varieties}

\begin{proposition}\label{proposition-ex-seq-general}
Let $X$ be smooth, projective and connected. 
We denote by $\eta$ the generic point of $X$. 
There is an exact sequence
\begin{multline*}
H^2_M(\eta,\Q(2))\xr{}  \Hom(\Q(-2),H^2(\eta,\Q)) \xr{} H^3_M(X,\Q(2))/H^3_M(X,\Q(2))_{{\rm dec}} \\ \xr{} H^3_{{\mc{H}}}(X,\Q(2))/H^3_{{\mc{H}}}(X,\Q(2))_{{\rm dec}}.
\end{multline*}
\begin{proof}
Via Proposition \ref{proposition-cohomology-Gersten} we identify 
\begin{align*}
H^1(G_M(X,2)) &\cong H^3_M(X,\Q(2)),\\
H^1(G_{HS}(X,2))&\subset H^3_{\mc{H}}(X,\Q(2))/H^3_{\mc{H}}(X,\Q(2))_{{\rm dec}}.
\end{align*}
We work with diagram \eqref{diagram-Gersten}. For the map 
$$
\Hom(\Q(-2),H^2(\eta,\Q)) \xr{} H^3_M(X,\Q(2))/H^3_M(X,\Q(2))_{{\rm dec}} 
$$
we observe that 
$$
0\xr{} \bigoplus_{x\in X^{(1)}} \C^*\otimes \Q \xr{} \bigoplus_{x\in X^{(1)}} H^1_M(x,\Q(1)) \xr{} \bigoplus_{x\in X^{(1)}} \Hom(\Q,H^1(x,\Q)(1)) \xr{} 0
$$
is exact and 
$$
H^3_M(X,\Q(2))_{{\rm dec}}={\rm im}(\bigoplus_{x\in X^{(1)}} \C^*\otimes \Q \xr{} H^3_M(X,\Q(2))).
$$
The assumptions on $X$ 
imply $H^0(G_{HS}(X,2))=0$ by
Proposition \ref{proposition-HS-Gersten-H0}. 
The rest of the proof
involves only diagram chasing.
\end{proof}
\end{proposition}

\begin{remark}
Proposition \ref{proposition-ex-seq-general} has also been proved by de Jeu and Lewis in \cite[Corollary~4.14]{JL}, and more generally with integral coefficients in \cite[Corollary~6.5]{JL}. 
\end{remark}

\begin{proposition}\label{proposition-reduction-to-generic-point}
Let $X$ be smooth, projective and connected. Suppose that $H^1(X,\Q)=0$.
The following statements are equivalent.
\begin{enumerate}
\item $BH(U,2)$ holds for all open subsets $U$ of $X$.
\item $BH(\eta,2)$ holds for the generic point $\eta$ of $X$.
\end{enumerate}
\begin{proof}
Only $(2)\Rightarrow (1)$ is interesting. Proposition \ref{proposition-ex-seq-general} 
implies that 
$$H^3_M(X,\Q(2))/H^3_M(X,\Q(2))_{{\rm dec}} \xr{} H^3_{{\mc{H}}}(X,\Q(2))/H^3_{{\mc{H}}}(X,\Q(2))_{{\rm dec}}$$
is injective. From Lemma \ref{lemma-injective-dec}, we see that 
$$
H^3_M(X,\Q(2))_{{\rm dec}} \xr{} H^3_{{\mc{H}}}(X,\Q(2))_{{\rm dec}}
$$ 
is an isomorphism. Thus 
$$
H^3_M(X,\Q(2)) \xr{} H^3_{{\mc{H}}}(X,\Q(2))
$$
is injective. It follows from Lemma \ref{lemma-reglobalisation} that 
$$
H^2_M(\eta,\Q(2)) \xr{} H^2_{{\mc{H}}}(\eta,\Q(2))
$$
is surjective. Lemma \ref{lemma-stupid-reduction-genpoint} applied to $U$
implies the claim.
\end{proof}
\end{proposition}
 
\begin{question}
Suppose $X$ is smooth, connected, projective, but $H^1(X)\neq 0$.
We denote by $\eta$ the generic point of $X$.
What is the relation between $BH(U,2)$, for all $U$ open in $X$, 
and the surjectivity of  
$$
H^2_M(\eta,\Q(2))\xr{} \Hom(\Q(-2),H^2(\eta,\Q))?
$$ 
\end{question}


\end{document}